\documentclass[a4paper]{article}
\usepackage{amsthm,amssymb,amsmath,enumerate,pgf,tikz,tikz-cd} 
\usepackage{graphicx,mathtools}

\usetikzlibrary{positioning,arrows,calc,matrix}
\usepackage[T1]{fontenc}
\usepackage[utf8]{inputenc}
\usepackage{authblk}

\title{ \huge {Inverse Limits and Topologies of Infinite Graphs} \thanks
{{\it Key Words}:  Compact spaces, Infinite graphs, Inverse limits.}
\thanks {2010{ \it Mathematics Subject Classification}:  05C10, 05C63, 57M15, 57M99.
 }}
\author{Babak Miraftab\thanks{babak.miraftab@uni-hamburg.de}}
\affil{Fachbereich Mathematik, Universit$\rm{\ddot{a}}$t Hamburg, Bundesstra{\rm{\ss}}e~$55$,~$20146$ Hamburg, Germany}

\tikzset{>=stealth}

\newcommand{\bbN}{\ensuremath{\mathbb{N}}}
\newcommand{\bbR}{\ensuremath{\mathbb{R}}}

\theoremstyle{definition}

\theoremstyle{plain}
\newtheorem{thm}{Theorem}

\newtheorem{lem}[thm]{Lemma}

\newtheorem{coro}[thm]{Corollary}
\newtheorem{remark}[thm]{Remark}

\newtheorem{deff}[thm]{Definition}

\theoremstyle{remark}

\newenvironment{txteq*}
  {
    \begin{equation*}
    \begin{minipage}[t]{0.85\textwidth} 
    \em                                
  }
  {\end{minipage}\end{equation*}\ignorespacesafterend}

\begin{document}
\maketitle
\begin{abstract}
Two of the natural topologies for infinite graphs with edge-ends are {\sc{Etop}} and {\sc{Itop}}. In this paper, we study and characterize them. We show that {\sc{Itop}} can be constructed by inverse limits of inverse systems of graphs with finitely many vertices. Furthermore, as an application of the inverse limit approach, we construct a topological spanning tree in {\sc{Itop}}.
\end{abstract}

\section{Introduction}
Studying graphs as topological spaces has a vast number benefits, see \cite{Georgakopoulos1,sur,mir,Polat,stein}.
This view allows us to compactify graphs.
For instance, considering infinite graphs as compact spaces enables us to define infinite cycles, see \cite{Georgakopoulos1}.
Compactifying infinite graphs is one of the controversial problems in infinite graph theory, see \cite{sur}. 

In 1931, Freudenthal \cite{Freu} introduced ends of  locally compact, connected, locally connected, $\sigma$-compact, Hausdorff  topological spaces~$X$ as points at infinity for compactification purposes. 
Essentially, Freudenthal’s ends are defined as descending sequences~$U_1\supseteq U_2 \supseteq \cdots$ of connected open sets with compact boundaries in a such way that~$\bigcap \overline{U_i}=\emptyset$.
Adding these ends with new appropriate open sets around them to $X$ leads to a new space which is compact.
This new compact space is called the \emph{Freudenthal compactification} of $X$.
In 1963, Halin  \cite{halin}, introduced graph-theoretical vertex-ends as equivalence classes of rays independently.
Those ends are, in general, distinct from Freudenthal’s ends.
In 2004, Diestel and K$\rm{\ddot{u}}$hn \cite{diestelkuhn} showed that these two kind of ends coincide for locally finite graphs.
More precisely, let~$G$ be a locally finite graph. Then the geometric realization of~$G$ is one-dimensional complex and we compactify~$G$ with the Freudenthal compactification and so some topological ends are obtained. 
Topological ends of~$G$ correspond to vertex-ends introduced by Halin.
It turns out that the Freudenthal compactification for locally finite graphs and the definition of infinite cycles as homeomorphic images of the unit circle~$S^1$ in the Freudenthal compactification of a graph are good approaches to extend extremal finite graph theory for infinite graphs, see \cite{sur}.

The most commonly used topologies on infinite graphs are {\sc Top},  {\sc{Etop}}, {\sc Itop},~$\ell$-{\sc Top}, {\sc Mtop}, {\sc Vtop}, for comprehensive details see~\cite{Georgakopoulos,diestel1}. 
It is worth mentioning that for a locally finite graph all those topologies coincide. 
So the importance of studying of them is when a given graph has a vertex of infinite degree.  
The topology {\sc{Etop}} was first defined by Diestel \cite{diestelper} though it appeared first in \cite{Schulz} by Schulz.  
Among of all non-trivial topologies for infinite graphs {\sc{Etop}} is the coarest. 
The ends considered in {\sc{Etop}} are edge-ends rather then the usual vertex ends. 
Note that {\sc{Etop}} is always a compact space, see \cite[Satz 2.1]{Schulz}.
The topology {\sc{Etop}} is not always Hausdorff. By identifying any two points that have the same open neighborhoods
and use the quotient topology  on {\sc{Etop}}, the  topological space {\sc{Itop}} is obtained.\footnote{Other names in the literature for {\sc{Etop}} and {\sc{Itop}} are {\sc{Etop}}$^{\prime}$ and {\sc{Etop}}, respectively, see  \cite{sur}.}
In this paper, we reconstruct the topologies {\sc{Etop}} and {\sc{Itop}} with different methods and show that all of them are  homeomorhic.
First,  we introduce a new topology for infinite graphs, namely {\sc{FCtop}}, with respect to edge-ends, which turns out to be equivalent to the topology {\sc{Etop}}. Then we define two families of inverse systems whose inverse limits are homeomorphic to {\sc{Itop}}.
Furthermore, as an application of our approach, we will construct  topological spanning tree for an infinite graphs with {\sc{Itop}} in Section 5.
 
\section{Preliminaries}
We refer readers to Diestel \cite{diestel} and Munkres \cite{munkres} for the standard terminologies and notations of graph theory and topology, respectively.
\subsection{Graphs}

Throughout this paper, graphs are infinite and connected and~$G$ will be reserved for graphs with the vertex set~$V(G)$ and the edge set~$E(G)$. 
A~$1$-way infinite path is called a \textit{ray}, a~$2$-way infinite path is a \textit{double ray}, and the subrays of a ray or double ray are its \textit{tails}. 
The union of a ray with infinitely many disjoint finite paths having precisely their first vertex is a \textit{comb} and the last vertices of these paths are \textit{teeth}. 
Two rays in a graph~$G$ are \textit{edge equivalent} if for any finite set~$F$ of edges,~$R_1$ and~$R_2$ have a tail in the same component of~$G$ without inner points of edges of~$F$. 
The corresponding edge equivalence classes of rays are the \textit{edge-ends} of~$G$ and for a ray $R$ and we show the corresponding edge-end by $[R]$.
We denote the set of all edge-end of~$G$ by~$\Omega'(G)$. 
It is important to notice that by replacing edge by vertex in the definition of edge-end, we obtain the \textit{vertex-end}, however in this paper  we are only concerned with edge-ends.
For distinguishing between vertex and edge ends see \cite{Hahn}.
For instance, let~$G$ be a graph as depicted on Figure 2.1.
Then~$G$ has exactly one edge-end and the vertex~$v$ dominates it. 
Note that we are not able to separate them by a finite cut.\\
\begin{center}
\begin{tikzpicture}
\draw[<->,>=latex'] (-0.5,1)--(9.5,1);
\draw[<->,>=latex'] (-0.5,5)--(9.5,5);
\draw (4.5,3) node[circle,fill, inner sep=2pt](v){}node[above=1pt]{$v$};
\foreach \i in {1,2,...,8}{
\draw (\i,1) node [circle,fill, inner sep=2pt] {};
}
\foreach \i in {1,2,...,8}{
\draw (\i,5) node [circle,fill, inner sep=2pt] {};
}
\foreach \i/\j in {0/1.2,1/1,2/1,3/1,4/1}{
\draw (v) edge[out=180,in=90] (\i,\j);
}
\foreach\i/\j in {5/1,6/1,7/1,8/1,9/1.3}{
\draw (v) edge[out=0,in=90] (\i,\j);
}
\foreach \i/\j in {0/4.8,1/5,2/5,3/5,4/5}{
\draw (v) edge[out=180,in=-90] (\i,\j);
}
\foreach\i/\j in {5/5,6/5,7/5,8/5,9/4.8}{
\draw (v) edge[out=0,in=-90] (\i,\j);
}
\draw (0.7,2) node(){$\mathbf\cdots$};
\draw (8.1,2) node(){$\mathbf\cdots$};
\draw (0.7,4) node(){$\mathbf\cdots$};
\draw (8.1,4) node(){$\mathbf\cdots$};
\begin{scope}[shift={(0.5,0.5)}]

\end{scope}
\end{tikzpicture}          
\end{center}
\begin{center}
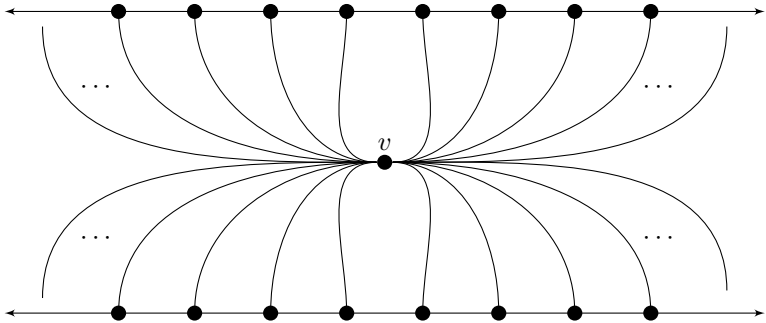
  Figure 2.1 a graph with only one edge-end
\end{center} 
For a given subset $A$ of vertices of $G$, we denote the induced subgraph with vertices of $A$ by $G[A]$.
Suppose that a pair~$(A,B)$ is a partition of the vertices of a graph into two disjoint subsets such that the number of edges between two sides is finite.
The set of these edges between $A$ and $B$ is called A \textit{finite cut}.
So we represent every finite cut~$C$ by a pair~$(A,B)$ where $A$ and $B$ are subsets of $V(G)$ such that $A\cap B=\emptyset$ and $V(G)=A\cup B$.
Thus with the above notation~$C$ is the set of edges which joins~$G[A]$ to~$G[B]$.
We note that the set of all finite cuts with empty forms a vector space over $\mathbb Z_2$.
We denote  \textit{finite cut space} by~$\mathcal{B}_{\rm{fin}}(G)$.
Let~$R$ be a ray. Then we say that a vertex~$v$ \textit{dominates}~$R$ if for any finite set~$F$ of the set of edges, there is~$v-R'$ path in~$G$ without inner points of edges of~$F$ where~$R'$ is a tail of~$R$. So a vertex dominates an end if it dominates the corresponding ray of this end. 
An edge-end~$\omega$ \textit{lives} in a component~$C$ of~$G$  if~$V[C]$ contains one ray belonging to~$C$ or equivalently each ray.  
Let~$F$ be a subset of~$E(G)$. Then by~$\mathring{F}$, we mean all inner points of edges of~$F$. 

\subsection{Topology}  
By a basic closed set, we mean the complement of a basic open set in a topological space.
For a set~$X$, we denote the power set of~$X$ by~$\mathcal{P}(X)$.
Let~$X$ be a space that is the union of the subspaces~$X_{\alpha}$, for~$\alpha \in I$.
The topology of~$X$ is said to be \textit{coherent} with the subspaces~$X_{\alpha}$ provided a subset~$C$ of~$X$ is closed in~$X$ if~$C\cap X_{\alpha}$ is closed in~$X_{\alpha}$ for each~$\alpha\in I$.
An equivalent condition is that a set be open in~$X$ if its intersection with each~$X_{\alpha}$ is open in~$X_{\alpha}$.
Now we move to topologies of graphs.
First the geometric realization of graphs is the one dimensional complex.\footnote{For the definition of the geometric realization, see \cite{munkres2}.}
We denote the geometric realization without considering its topology of~$G$ by~$\Vert G\Vert$.
So we are able to regard inner points of edges  of a graph~$G$ as points of~$\Vert G\Vert$. 
For defining {\sc{Etop}} on~$\Vert G\Vert\cup \Omega'(G)$, we describe open sets. 
For each~$e\in E(G)$,~$\mathring e$ inherits the topology of open interval~$(0,1)$.
For any finite set~$F$ of edges of~$G$, we remove a finite set~$X$ of inner points of edges of~$F$.
Suppose that~$C$ is a component of~$\Vert G\Vert\setminus X$.
Then~$C\cup\{\omega\in\Omega'(G)\mid \omega\text{ lives on }C\}\cup L$ forms a open set, where~$L$ is the set of all partial edges like~$[a,b)$  and~$b$ is the inner point which picked up from the edge~$ac\in F$ with
$a$ lying in~$C$. 
With a similar method, we can define when~$c$ lies in~$C$ with replacing~$(b,c]$ with~$[a,b)$.
We denote the open set around an end~$\omega$ with respect to a finite set~$X$ of~$E(G)$, by~$\mathcal{O}_X(\omega)$.
The topology generated by these open sets is called {\sc Etop}.
For a locally finite graph~$G$, {\sc Etop} and  the Freudenthal compactification \cite{Freu} of the 1-complex of the graph~$G$ are the same, see \cite{diestelkuhn}.\\
\noindent It is worth noting that~$(\Vert G\Vert\cup\Omega'(G)$, {\sc Etop})  is not Hausdorff.
The solution for obtaining a Hausdorff space is identifying any two points that have the same open neighborhoods.
In other words, we define an equivalence relation between points i.e. for two points~$x,y\in\Vert G\Vert\cup\Omega'(G)$, we define~$x\sim y$ if and only if we cannot separate~$x$ and~$y$ with a finite subset of edges.
For instance, every dominating vertex is equivalent with the corresponding edge ends, see Figure 2.1.
Then we use the quotient topology and obtained a new topological space 
$\widetilde{G}$. We denote this topology by {\sc Itop}. 
Strictly speaking, {\sc Itop} is not a topology for an infinite graph, as we are identifying some points.\\

\noindent For defining our topologies, we need inverse systems and inverse limits.
Since these terminologies are one of central notations of this paper, let us review here.
Let~${(I,\preceq)}$ denote a directed poset, that is, a set with a binary relation~$\preceq$ satisfying reflexivity, antisymmetry, transitivity and moreover if~$i,j\in I$ there exists some~$k\in I$ such that~$i,j\preceq k$. 
An \textit{inverse system} of topological spaces over~$I$ consists of a collection~$\{X_i\mid i\in I\}$ of topological spaces indexed by~$I$ and a collection of continuous maps~$f_{ij}\colon X_i\to X_j$ defined whenever~$i\preceq j$ such that the diagrams of the form
\begin{center}
\begin{tikzcd}[column sep=small]
 X_i \arrow{rr}{f_{ij}} \arrow[swap]{dr}{f_{ik}}& & X_j \arrow{dl}{f_{jk}}\\
& X_k & 
\end{tikzcd}
\end{center}
\begin{center}
Figure 2.2
\end{center}
commute whenever they are defined, i.e., whenever~$i, j, k \in I$ and~$i\preceq j\preceq k$.
In addition we assume that~$f_{ii}$ is the identity mapping~$id_{X_i}$ on~$X_i$. We denote this inverse system over~$I$ by~$(X_i,f_{ij},I)$.\\
Now, assume that~$Y$ be a topological space and~$g_i\colon Y\to X_i$ is a continuous map for each~$i\in I$.
The maps~$g_i$s are called compatible if~$f_{ij}\circ g_i =g_j$ for every~$i,j\in I$.
A topological space~$X$  with compatible continuous map~$f_i\colon X\to X_i$ for~$i\in I$
is called an \textit{inverse limit} of the inverse system~$(X_i ,f_{ij} , I)$, if
there is a unique continuous map~$f\colon Y\to X$  satisfying~$f_i\circ f=g_i$.

\noindent For comprehensive detail about the inverse limit of topological spaces, see \cite{profinite group}. The following lemma plays a vital role in this paper. In fact this is an immediate corollary of \cite[Lemma 1.1.2]{profinite group}.
\begin{lem}\label{compact}
If~$(X_i,f_{ij},I)$ is an inverse system of compact Hausdorff topological
spaces, then~$\underleftarrow{\lim}\,X_i$ is compact. 
\end{lem}


\section{New Topologies}
In this section, we define a new topology for infinite graphs with edge-ends. 
To define this topology,
we use finite cuts and instead of defining basic open sets for each point of~$\Vert G\Vert\cup \Omega^{\prime}(G)$, we introduce basic closed set for them. Then we introduce two new topological spaces. 
In order to introduce these new topological spaces, we define two families of auxiliary graphs with two different methods and we show that they constitute inverse systems.
\noindent We start with the definition of a new topology for infinite graphs.\\

\noindent First for any edge~$e$ of~$G$,~$\mathring{e}$ is endowed by the open interval~$(0,1)$. 
For any finite cut~$C=(A,B)$ of~$G$, we remove~$\mathring{C}$ of~$G$. 
We now define every component of~$ G\setminus \mathring C$ as basic closed set with respect to~$C$.
We need to define a basic closed set for a given end $\omega$. 
A basic closed around an end~$\omega$ is~$\mathcal{C}_C(\omega)=F\cup\{\omega\in\Omega'(G)\mid \omega\text{ lives on }C\}$ where~$F$ is the unique component which~$\omega$ lives in it.
We call the above topology~{\sc{FCtop}}.  
It is worth mentioning that after removing~$\mathring{C}$, we will have a finite number of components.\\ 
Recall that for defining {\sc{Itop}}, we identified any two points that have the same open neighborhoods in {\sc{Etop}}. 
Equivalently, we used an equivalence relation between vertices so that for two vertices we have~$x\sim y$ if and only if we cannot separate~$x$ and~$y$ with a finite cut. 
Also if we have an end which is dominated by a vertex, see Figure~$2.1$, then we identify them. 
Now let us get back to our definition. 
We need to get a Hausdorff space, but there might be some vertices which do not have any separation by finite cuts and the same problem like {\sc{Etop}} for dominating  vertices by some edge ends.
We identify these points by defining an equivalence relation on~$\Vert G\Vert\cup \Omega^{\prime}(G)$. 
Now we use the quotient topology on this quotient space. 
We denote this new space obtained by taking quotient of~$\Vert G\Vert\cup \Omega'(G)$ by the equivalence relation and the quotient topology on it with~$\widetilde{G}$ and {\sc{IFCtop}}, respectively.\\

\noindent To show that~{\sc{FCtop}} is compact, we need the following famous lemma namely star-comb lemma.
\begin{lem}{\rm\cite[Lemma 8.2.2]{diestel}}\label{star-comb}
Let~$U$ be an infinite set of vertices of a connected graph~$G$. Then~$G$ contains either a comb with all teeth in~$U$ or a subdivision of an infinite star with all leaves in~$U$. 
\end{lem}
\begin{thm}\label{fctopcompact}
If~$G$ is a countable graph. Then~$(G,{\textit{\sc{FCtop}}})$ is a compact space.
\end{thm}
\begin{proof}
In order to show the compactness of~$(G,{\textit{\sc{FCtop}}})$, we take any collection of basic closed sets~$\{C_{i}\}_{i\in \mathbb N\cup\{0\}}$ with the finite intersection property and then we show that the intersection of this collection is not empty. 
We note that since $G$ is countable, there are countably many basic closed sets. 
Let~${x_{0}\in C_{0}}$.
Then we can find a point~$x_{i}\in C_{0}\cap\cdots\cap C_{i}$. 
Let~$U$ be the collection of all~$x_{i}$'s with the above property.
It follows from Lemma \ref{star-comb} that~$G$ contains either a comb with all teeth in~$U$ or a subdivision of an infinite star with all leaves in~$U$. 
First suppose that we have a ray~$R$ with all teeth in~$U$.
We claim that the end~$[R]$ is included~$\bigcap C_{i}$.
If every~$C_i$ contains a tail of a ray in~$[R]$, then we are done.
So assume to contrary that~$C_{k}$ has no tail of a ray in~$[R]$.
Then there are infinitely many vertices of~$R$ outside of~$C_{k}$.
We denote them by~$Y$. It follows from the choice of~$x_{i}$ that there is an infinite subset~$\Lambda$ of~$\mathbb N\cup \{0\}$ such that~$x_{i}\in C_{k}$ for any~$i\in \Lambda$.
Let~$X:=\{x_{i}\}_{i\in\Lambda}$.
It is not hard to see that there are infinitely many disjoint~$X$-$Y$  paths.
On the other hand,~$C_{i}$ is a basic closed set which is separated by finitely many edges.
It yields a contradiction with infinitely many disjoint~$X$-$Y$  paths from the outside of~$C_{k}$ to the inside of~$C_{k}$. So the claim is proved.

\noindent Now suppose that~$G$ contains a subdivision of an infinite star with all leaves in~$U$. 
Let~$v$ be the center of this infinite star. 
We show that~$v$ belongs to~$\bigcap C_{i}$. 
Again there is~$C_{k}$ such that~it does not contain~$v$. 
There are infinitely many~${i\in \mathbb N\cup \{0\}}$ such that~$x_{i}\in C_{k}$. 
Hence~$v$ has infinitely many leaves in~$C_{k}$ and it contradicts with being basic closed of~$C_{k}$. 
Thus~$\bigcap C_{i}$ is not empty and we deduce that~$(G,{\textit{\sc{FCtop}}})$ is compact, as desired. 
\end{proof}
\noindent A graph~$G$ is said \textit{finitely separable} if  every two vertices can be separated by some finite set of edges.
\begin{coro}
Let~$G$ be a finitely separable~$2$-connected. Then~$(G,{\textit{\sc{FCtop}}})$ is  compact.
\end{coro}
\begin{lem}\label{uryson}{\rm\cite[Theorem 34.1]{munkres}}{\rm(Urysohn metrization theorem)} Every regular space with a countable basis is metrizable.
\end{lem}
\noindent Note that every Hausdorff compact space is normal and so it is regular.
Now let $G$ be a countable graph.
Theorem \ref{fctopcompact} implies that ($\widetilde{G}$, {\sc{IFCtop}}) is a regular space and by theorem \ref{uryson}, we have the following theorem.
\begin{thm}\label{metrizable}
Let $G$ be a countable graph. Then ($\widetilde{G}$, {\sc{IFCtop}}) is metrizable.
\end{thm}
\noindent In the following, we introduce the first family of inverse systems. 
First, we define a family of auxiliary graphs with finitely many vertices. 
Next we study these auxiliary graphs and their connection with the primary graph. 
The following auxiliary graphs were defined for the first time in \cite{mir} for extending and generalizing flow theory of finite graphs to infinite graphs.

\noindent We can imagine our auxiliary graphs in the following way:\\
We consider a partition $\{V_1,\ldots,V_t\}$ of~$G$ such that there are only finitely many cross-edges between these $V_i$'s.
Then we contract all vertices of each partition to a single vertex, but we keep the edges.
In other words, every partition with the above property gives a multi-graph with finitely many vertices.
Next we define these partitions more precisely.
\begin{deff}\label{contraction 2}
Let~$M=\{C_{1},\ldots,C_t\}$ be a finite subset of the space~$\mathcal{B}_{\rm{fin}}(G)$, where~$C_i=(A_i,B_i)$.
Define~$$V(\mathcal{G}_M)=\left\{ X_1\cdots X_t\mid X_{i}\in\{A_{i},B_{i}\}\,\textit{and}\, \bigcap_{i=1}^t X_i\neq\emptyset \right\}.$$
We add edges between~$U_1=X_{1}\cdots X_t$ and~$U_2=X'_{1}\cdots X'_t$  for every edge between 
$\cap_{i=1}^t X_{i}$ and~$\cap_{i=1}^t X'_{i}$ in the original graph~$G$. 
\end{deff}
\noindent Every vertex~$X_1\cdots X_t\in V(\mathcal{G}_M)$ defines with $\bigcap_{i=1}^tX_i$ a subset of $V(G)$ and $V(\mathcal G_M)$ admits a partition of $V(G)$ and there are finitely many cross-edges between these partition, as we said before the definition.\\
For simplicity, we define~${\Phi(U)=\cap_{i=1}^t X_{i}}$, for every vertex~$U=X_{1}\cdots X_t$ of~$\mathcal{G}_M$.\\

\noindent For a given finite subset~$M$ of~$\mathcal{B}_{\rm{fin}}(G)$, we define a natural map~$\phi_M$ from ${\Vert{G}\Vert\cup\Omega'(G)}$ to $\Vert\mathcal{G}_M\Vert$.
First, we define~$\phi_M$ on the set of vertices of~$G$. For every vertex~$u\in V$, we associate a unique vertex~$U$ of~$\mathcal{G}_M$. 
Consider a finite cut~$C=(A,B)$ in~$M$. 
By the definition of~$C$, either~$A$ or~$B$ 
should contain~$u$ and we do this for every finite cut in~$M$. 
Let~$X_i$ be the suitable part containing~$u$ for 
$i=1,\ldots,t$. We define~$\phi_M(u)=X_1\cdots X_t$. 
For edge-ends, we can do the same. 
For a given end~$\omega$, one of~$A$ or~$B$ of a cut~$(A,B)$ should contain a tail of a ray corresponding to the end~$\omega$ and with using an analogous argument, we can build up the unique word containing~$\omega$. 
Now, it is clear how we have to define the set of edges. 
If~$uv\in E(G)$, then~$\phi_M(uv)=UV$, where~$U$ and~$V$ are the corresponding vertices to~$u$ and~$v$ respectively.
Thus we have the following lemma.
\begin{lem}\label{contraction map}
With the above notations the following holds:
\begin{enumerate}[\rm (i)]
\item The map~$\phi_M$ is surjective.
\item The restriction of~$\phi_M$ from~${E(G)}\cup \mathring{E}(G)$ to~${E(\mathcal{G}_M)}\cup \mathring E(\mathcal{G}_M)$ is  identity.
\end{enumerate}
\end{lem}

\noindent  It is worthy to mention that~$\phi_M^{-1}(V)\cap\phi_M^{-1}(U)=\emptyset$ for vertices~$U\neq V$ of~$\mathcal {G}_M$. Therefore we get a partition for vertices of~$G$.\\

\noindent In the following, we topologize~$\Vert\mathcal{G}_M\Vert$. 
First we discuss the case when~$G$ is a countable graph.
In order to put our topology on~$\Vert\mathcal{G}_M\Vert$, we need to define topologies for three special subgraphs.
The first one is well-known as the Hawaiian earring or infinite earring see \cite[Example 1 of page 436]{munkres}, 
the second one is the finite version of the Hawaiian earring and the last graph is~$K_2$.
The Hawaiian Earring is defined to be a subspace of~$\bbR^2$ consisting of the union of planar circles~$c_i$ of radius~$1/i$, tangent to the x-axis at the origin for~$i\in\bbN$.
This space is well-known to have interesting properties, see \cite{can}. 
Note that the Hawaiian earring is a closed and  bounded subset of~$\bbR^2$ and so it is compact; but its fundamental group is uncountable.
The finite version of the Hawaiian earring is a subspace of~$\bbR^2$ consisting of the union of planar circles~$c_i$ of radius~$1/i$, tangent to the $x$-axis at the origin for~$i\in N$ where~$N$ is a finite subset of~$\bbN$.
We consider the same topology for the finite version of the Hawaiian earring.
Moreover, every~$K_2$ is endowed with the topology of the closed interval~$[0,1]$. 
We denote the subgraphs which are homeomorphic to the Hawaiian earring and a finite version of Hawaiian earring by~$\mathcal{H}_{\aleph_0}$ and~$\mathcal{H}_N$, respectively. \\

\begin{center}
\begin{tikzpicture}
\draw (0,-2) node [circle,fill, inner sep=2pt] {};
      \draw[solid] (0,-1.75) circle (0.25);
 \draw[solid] (0,-1.5) circle (0.5);
    \draw[solid] (0,-1) circle (1);
    \draw[solid] (0,0) circle (2);
    \draw  (0,-1.69) node(){$\mathbf\vdots$};

\end{tikzpicture}
\end{center}
\begin{center}
Figure 3.1 The Hawaiian earring
\end{center}
\noindent Now, we are ready to topologize~$\Vert\mathcal{G}_M\Vert$ for a given finite subset~$M$ of~$\mathcal{B}_{\rm{fin}}(G)$. 
Since~$\Vert\mathcal{G}_M\Vert$ is a union of finitely many of copies of the subgraphs~$\mathcal{H}_{\aleph_0}$,~$\mathcal{H}_N$ and~$K_2$, we will not have any problem to use the coherent topology here, i.e.~$U$ is an open (closed) set in~$\Vert\mathcal{G}_M\Vert$ if and only if the intersection of~$U$ with each of these subspaces is an open (closed) set. 
Therefore we obtain a compact space, as~$\Vert\mathcal{G}_M\Vert$ is a union of finitely many compact spaces. 
Note that~$\Vert\mathcal{G}_M\Vert$ is a Hausdorff space as well. 
We denote this topology by~$\tau_M$, as~$\Vert\mathcal{G}_M\Vert$ is constructed by the finite set of cuts~$M$. \\
 
\noindent Now, let~$G$ be an uncountable graph. 
Then we cannot embed the Hawaiing earring in~$\bbR^2$, i.e.~we are not able to embed uncountable many loops in $\mathbb R^2$ and use the induced topology of~$\bbR^2$. 
Now we define a new topological space for the uncountable version of the Hawaiing earring.\footnote{For the sake of simplicity, we use a closed interval whose the cardinal is $2^{\aleph_0}$, however the method works for any compact space with an arbitrary cardinal $\kappa$.} 
Suppose that~${I=[0,1]}$.
Fix a point of~$S^1$ say~$x_0$. Consider the quotient space~$(S^1\times I)/(\{x_0\}\times I)$. 
So we have an injective map~$\iota$ from this 
space to the product of~$S^1$,~$I$ times. Let~$X_i$ be homeomorphic to~$S^1$ for every~$i\in I$.  Then we have the following map.
\begin{align*}
 \iota\colon &\frac{S^1\times I}{\{x_0\}\times I}\longrightarrow \prod_{i\in I} X_i \\
& (x,i)\xmapsto{\phantom{L^\infty(T)}} \left(j\to \left\{ \begin{array}{lcl}
x_0 & \mbox{for}
& i\neq j \\ x & \mbox{for} & i=j 
\end{array}\right.\right)
\end{align*}
We claim that the image~$\iota$ i.e. $\{(x_j)_{j\in I}\mid \exists \textit{ at most one }j\in I\textit{ with }x_j\neq x_0 \}$ is a closed set in~$\prod_{i\in I} X_i$. 
Assume that~$(x_{\alpha})_{\alpha\in I}$ is an element of the complement of~${\rm{Im}}{(\iota)}$. 
So there are indices~$\alpha_0$ and~$\alpha_1$ in~$I$ such that both of~$x_{\alpha_0}$ 
and~$x_{\alpha_1}$ are not~$x_0$. We can find open sets~$O_{\alpha_0}$ and~$O_{\alpha_1}$ containing~$x_{\alpha_0}$ 
and~$x_{\alpha_1}$ without~$x_0$, respectively. Now put~$O:=\prod O_{i}$ 
where~$O_{i}=X_i$ for any~$i\neq \alpha_1,\alpha_2$. It is not hard to see that~$O$ is an open set of the complement of~${\rm{Im}}{(\iota)}$ containing~$(x_{\alpha})_{\alpha\in I}$ and the claim is proved. Since~${\rm{Im}}{(\iota)}$ is a closed subset of a compact space, it is compact.
We now define the~${\rm{Im}}{(\iota)}$ as the uncountable Hawaiing earring and we denote it by $\mathcal{H}_{2^\aleph_0}$.
With an analogous method we define $\mathcal{H}_{\kappa}$, for an arbitrary cardinal $\kappa$.\\
Let us get back to our objective which is defining a topology for~$\mathcal{G}_M$. 
Whenever~$G$ is uncountable, we benefit from $\mathcal{H}_{\kappa}$ for a suitable $\kappa$ and like the above case we obtain a compact Hausdorff topology for~$\Vert\mathcal{G_M}\Vert$.\\
  
\noindent Now we summarize all the above discussion on the following theorem.
\begin{thm}\label{G_M compact}
Let~$M$ be a finite subset of~$\mathcal{B}_{\rm{fin}}(G)$. Then~$(\Vert\mathcal{G}_M\Vert,\tau_M)$, is a  compact Hausdorff space.  
\end{thm}

\noindent Recall that we defined the map~$\phi_M$ from~$(\Vert{G}\Vert\cup\Omega'(G),{\textit{\sc{FCtop}}})$ to~$(\Vert\mathcal{G}_M\Vert, \tau_M)$ for a finite subset~$M$ of~$\mathcal{B}_{\rm{fin}}(G)$.\\
\noindent The next theorem reveals the relationship between the space~$(\Vert{G}\Vert\cup\Omega'(G),{\textit{\sc{FCtop}}})$ with the space~$(\Vert\mathcal{G}_M\Vert,\tau_M)$. 
\begin{thm}\label{fctopinverselimit}
The map~$\phi_M\colon(\Vert{G}\Vert\cup\Omega'(G),{\textit{\sc{FCtop}}})\to(\Vert\mathcal{G}_M\Vert, \tau_M)$ is  continuous.
\end{thm}
\begin{proof}
First, assume that~$I$ is an open set around an inner point~$x$ of an edge~$e$. 
Without loss of generality, we can assume that~$I$ is an open interval and $I\subset e$. 
Since~$\phi_M$ is the identity map from $E(G)\cup \mathring{E}(G)$ to ${E(\mathcal{G}_M)}\cup \mathring E(\mathcal{G}_M)$, the preimage of~$I$ is an open set in~${\Vert G\Vert}$. 
So~$\phi^{-1}_M(I)$ is open in~$\Vert{G}\Vert\cup\Omega'(G)$.
Now let~$O$ be a basic open set containing a vertex~$v$ of~$\mathcal{G}_M$. 
If the degree of~$v$ is finite, then the preimage is a union of some vertices and some parts of edges~$\{e_1,\ldots,e_{m-1}\}\cup\{e_1',\ldots,e_t'\}$ and entire edges~$\{e_m,\ldots,e_n\}$, where the boundary of $\overline{\phi^{-1}_M(O)}$ touches each $e_i$ twice and each $e_j'$ once for $i=1,\ldots,m-1$ and $j=1,\ldots,t$, see Figure 3.2.
We may suppose that the set~$\{e_1',\ldots,e_t'\}$ forms a finite cut which can be presented by~$C=(A,B)$. 
In addition we can assume that the preimage of~$O$ is contained in~$G[A]\cup C$. 
We show that the complement of~$\phi^{-1}_M(O)$ is a closed set. 
Note that~$G[B]$ is a union of finitely many components.
For obtaining the complement of~$O$, we need to add some partial edge to the closure of~$G[B]$ which are some parts of edges~$e_1',\ldots,e_t'$ and we might need to add some parts of the rest of edges. 
But all of these partial edges are closed with~{\sc{FCtop}}, see Figure 3.2.
\begin{center}
\begin{tikzpicture}
\draw (-3,-2) node [circle,fill, inner sep=2pt]{};
\node at ($(-3,-2.3)$){$v$};
\node at ($(-4.3,-2.3)$){$e'_1$};
\node at ($(-4.1,-2.9)$){$e'_2$};
\node at ($(-1.8,-2.7)$){$e'_t$};
\node at ($(-3.5,0.5)$){$e_1$};
\node at ($(-3.4,0.1)$){$e_2$};
\node at ($(-3.1,-0.7)$){${\scriptstyle e_{m-1}}$};
\draw[solid] (-3,-1.850) circle (0.15);{e}
\draw  (-3,-0.3) node(){$\mathbf\vdots$};
\draw (-3,-0.8) ellipse (1.1cm and 1.2cm);
\draw[solid] (-3,-1) circle (1);
\draw (-3,-1.4) ellipse (0.33cm and 0.6cm);
\draw (-3,-2)--(-4.2,-2.6) ;
\draw (-3,-2)--(-2,-2.5) ;
\draw (-3,-2)--(-3.9,-3);
\draw  (-2.98,-2.55) node(){$\mathbf\cdots$};
\draw (-3,-1.66) ellipse (0.2cm and 0.35cm);
\draw  (-3,-1.40) node(){$\mathbf\vdots$};
\draw[thick,dashed] (-3,-2) circle (0.75);

\node at ($(-1,-1.5)$){$\Longrightarrow$};

\draw[thick,dashed] (3.5,-1.5) ellipse (3cm and 1.5cm);
\draw(2,-1) ellipse (1cm and 0.5cm);
\draw(4.5,-1) ellipse (1.3cm and 0.7cm);
\draw(2.5,-2.2) ellipse (1.3cm and 0.5cm);
\draw (1.5,-0.8) node[circle,fill, inner sep=2pt](w){}node[below=1pt]{$v_{12}$};
\draw (4.5,-0.6) node[circle,fill, inner sep=2pt](z){}node[below=1pt]{$v_{11}$};
\draw (1.5,-2.2) node[circle,fill, inner sep=2pt](u){}node[right=1pt]{$v_{22}$};
\draw (w) edge[out=130,in=90] (z);
\draw (u) arc (45:310:0.8cm and -1cm) (w);
\draw (2.5,-2.5) node [circle,fill, inner sep=2pt]{}node[above=1pt]{$u_2$};
\draw (2.5,-2.5)--(2.2,-3.6);   
\draw (2,-0.8) node [circle,fill, inner sep=2pt]{}node[below=1pt]{$u_1$};
\draw (2,-0.8)--(1.5,0.3); 
\draw (5,-0.8) node [circle,fill, inner sep=2pt]{}node[below=1pt]{$u_t$};
\draw (5,-0.8)--(5.5,0.3);
\draw (5.1,-2.1) node [circle,fill, inner sep=2pt]{}node[above=1pt]{$v_{m-11}$};
\draw (3.2,-2.2) node [circle,fill, inner sep=2pt]{}node[below=1pt]{$v_7$};
\draw (5.5,-1.3) node [circle,fill, inner sep=2pt]{}node[below=1pt]{$v_{m-12}$};
\draw (5.1,-2.1)--(3.2,-2.2);
\draw (5.1,-2.1) arc (180:-125:0.8cm and -1cm);
\draw(5,-2.2) ellipse (0.7cm and 0.3cm);
\node at ($(4,-1.99)$){$e_m$};
\node at ($(2.7,0.4)$){$e_1$};
\node at ($(-0.1,-1.5)$){$e_2$};
\node at ($(1.40,0.2)$){$\scriptstyle e_1'$};
\node at ($(2.1,-3.3)$){$\scriptstyle e_2'$};
\node at ($(5.2,0.2)$){$\scriptstyle e_t'$};
\node at ($(6.2,-3.3)$){$ e_{m-1}$};
\node at ($(-3,-3.1)$){$O$};
\node at ($(3.5,-3.5)$){$\phi^{-1}(O)$};
\end{tikzpicture}
\end{center}
\begin{center}
Figure 3.2
\end{center}
More precisely, each of~$\mathring{e}_i'$ can be regarded as the union of~$(u_i,w_i)\cup[w_i,z_i)$ so that~$(u_i,w_i)$ is included in~$O$. 
In addition, every~$\mathring e_i$ for~$i=1,\ldots,m-1$ is divided by three intervals~$(v_{i1},w_{i1})$,~$[w_{i1},w_{i2}]$ and~$(w_{i2},v_{i2})$ such that we have~${(v_{i1},w_{i1})\cup(w_{i2},v_{i2})\subseteq O}$ and also~$\{e_{m+1},\ldots,e_n\}$ is included in~$O$. 
Thus the complement of~$\phi^{-1}(O)$ is~$$D\cup\bigcup_{i=1}^{m-1}[w_{i1},w_{i2}]\cup\bigcup_{i=1}^t[w_i,z_i),$$
where~$D$ is the union of~$G[B]$ with all ends which live in~$B$. 
Let~$B=C_1\cup\cdots\cup C_l$ where~$C_i$ is a component of~$G\setminus \mathring{C}$. 
Let an end~$\omega$ live in~$C_i$. 
Note that the union~$C_i$ with all ends which live in~$C_i$ is~$\mathcal{C}_C(\omega)$ that is closed by definition. 

\noindent Now suppose that the degree of~$v$ is infinite. 
So there is an infinite Hawaiian earring which occurs at~$v$. 
It is important to notice that there are only finitely many edges incident to~$v$ meeting~$O$. 
In this case there are infinitely many edges inside of~$\phi^{-1}_M(O)$ in spite of the last case. 
Hence with using a similar method which we used in the preceding case, we can show that the complement of~$\phi^{-1}_M(O)$ is a closed set in~$\Vert{G}\Vert\cup\Omega'(G)$.
\end{proof}

\begin{lem}\label{closed}
The map~$\phi_M\colon (\Vert{G}\Vert\cup\Omega'(G), {\textit{\sc{FCtop}}})\to(\Vert\mathcal{G}_M\Vert, \tau_M)$ is closed.
\end{lem}

\begin{proof}
Let~$K$ be a basic closed set around a vertex~$v$ of $G$. 
The image of~$K$ by~$\phi_M$ contains finitely many vertices. 
More precisely,~$\phi_M(K)$ is a union of finitely many Hawaiian earrings, finite Hawaiian earrings and finitely many copies of~$K_2$. 
Note that it does not contain any partial edges and so it is closed. 
Let~$K$ be a closed set around an inner point~$x$ which is included in an edge~$e$. 
Then since~$\phi_M$ is identity on~${E(G)}\cup\mathring{E}(G)$, the set~$\phi_M(K)$ is closed. 
Now let~$\omega$ be an edge-end and let~$\mathcal{C}_C(\omega)$ be an arbitrary basic closed set around~$\omega$ with respect to the finite cut~$C$ in~$M$. 
We show that~$\phi_M(\mathcal{C}_C(\omega))$ is closed in~$(\Vert\mathcal{G}_M\Vert, \tau_M)$. 
The image of~$\mathcal{C}_C(\omega)$ contains finitely many vertices. 
Again~$\phi_M(\mathcal{C}_C(\omega))$ is a union of finitely many Hawaiian earrings, finite Hawaiian earrings and some copies of~$K_2$. 
Note that it does not contain any partial edges. 
Thus the image of the basic closed set~$\mathcal{C}_C(\omega)$ is closed, as desired.  
\end{proof}

\noindent We accomplished to study the connection between topological spaces~$(\Vert{G}\Vert\cup\Omega'(G),{\textit{\sc{FCtop}}})$ and~$(\Vert\mathcal{G}_M\Vert,\tau_M)$. 
Next we investigate the graph~$\mathcal{G}_M$ for different finite subsets~$M$ of~$\mathcal{B}_{\rm{fin}}(G)$.

\begin{lem}\label{M to M'}
Let~$M\subseteq M'$ be two finite subsets of~$\mathcal{B}_{\rm{fin}}(G)$. Then there exists a continuous map~$\psi_{M'M}\colon(\Vert \mathcal{G}_{M'}\Vert,\tau_{M'})\to(\Vert\mathcal{G}_M\Vert,\tau_M)$.
\end{lem}

\begin{proof}
Let~$V'\in V(\mathcal{G}_{M'})$. 
Then~$\Phi(V')$\footnote{For the definition of~$\Phi$, see after Definition \ref{contraction 2}.} is a subset of~$V(G)$ which is obtained by elements of~$M'$. 
On the other hand, every element of~$M$ is an element of~$M^\prime$. 
Hence we can find a word~$W$ in~$\mathcal{G}_M$ such that~$\Phi(V')$ is included in~$\Phi(W)$. 
Thus we assign the vertex~$V'$ of~$\mathcal{G}_{M'}$ to the vertex~$W$ of~$\mathcal{G}_M$. 
Note that every~$\mathcal{G}_M$ contains all edges of~$G$. 
So we have a map~$\psi_{M'M}\colon \mathcal{G}_{M'}\to \mathcal{G}_M$, where~$\psi$ carries~$V'$ to~$W$ and every edge~$e$ to~$e$. 
We now show that the map~$\psi_{M'M}$ is continuous. 
Since~$M$ is subset of~$M'$, the partition of $\Phi(V(\mathcal{G}_{M'}))$ is a finer partition of $V(G)$ then $\Phi(V(\mathcal{G}_{M}))$.
In fact we contract it to smaller pieces and the contraction is a continuous map such that the new partitions are contained in the old partitions. 
More precisely, let~$K$ be a closed set in~$(\Vert\mathcal{G}_M\Vert,\tau_M)$. 
Since~$\phi_M$ is continuous,~$\phi_M^{-1}(K)$ is closed in~$(\Vert{G}\Vert\cup\Omega'(G), {\textit{\sc{FCtop}}})$ and it follows from Lemma \ref{closed} that~$\phi_{M'}(\phi_M^{-1}(K))$ is closed in~$(\Vert\mathcal{G}_{M'}\Vert, \tau_{M'})$. 
So it is enough to show that~$\psi_{M'M}\circ\phi_{M'}=\phi_M$. 
Since all maps are the identity on edges, it is enough to show equality for vertices and ends. 
By definition of~$\phi_{M'}$, every vertex~$v$ of~$G$ maps to the unique vertex~$U$ of~$\mathcal{G}_M$ and it follows from definition of~$\psi_{M'M}$ that the image of~$U$ by~$\psi_{M'M}$ is exactly the same as the image~$v$ by~$\phi_M$.  
Now let~$\omega$ be an end of~$G$. 
Then with regarding to the cuts of the set~$M'$, we can build up the unique word~$U$ which is a vertex of~$\mathcal{G}_{M'}$.\footnote{For building up the word, see the proof of Lemma \ref{contraction map}.} 
Again by the definition of~$\psi_{M'M}$ the image of~$U$ by~$\psi_{M'M}$ is equal to the image of~$\omega$ by~$\phi_M$. 
Hence~$\psi_{M'M}$ is continuous, as desired.
\end{proof}

\noindent As a consequence of  Lemma \ref{compact} and the previous lemma, we have the following theorem.

\begin{thm}\label{compatible}
 The system~$(\Vert\mathcal{G}_{M}\Vert,\psi_{MM'},\Gamma )$ is an inverse system, where~$\Gamma$ is the set of all finite subset of~$\mathcal{B}_{\rm{fin}}(G)$ and~$M,M'\in \Gamma$ and moreover the space~$\underleftarrow{\lim}\,\Vert\mathcal{G}_M\Vert$ is a compact Hausdorff space. 
\end{thm}

\begin{proof}
In Lemma \ref{M to M'}, we show that~$\psi_{MM'}\colon\Vert\mathcal{G}_M\Vert\to\Vert\mathcal{G}_{M'}\Vert$ for~$M'\subseteq M$ is continuous.  Let~$M_1\subseteq M_2\subseteq M_3$ be three finite subsets of~$\mathcal{B}_{\rm{fin}}(G)$. We have to show that $\psi_{M_1M_3}=\psi_{M_2M_3}\circ \psi_{M_1M_2}$.
Note that every vertex of~$\mathcal{G}_{M_1}$ admits a partition of the set of vertices of~$G$. 
In fact we contract this component to this vertex. 
Since~$M_1\subseteq M_2\subseteq M_3$,  we can deduce that~$\Phi(V(\mathcal G_{M_i}))$ is a finer partition  than~~$\Phi(V(\mathcal G_{M_{i-1}}))$, for~$i=2,3$. 
Suppose that $\{V_1,\ldots,V_t\}=V(\mathcal{G}_{M_2})$. 
Let $\{V_{i_1},\ldots,V_{i_{\ell}}\}$ be all vertices of $\mathcal G_{M_2}$ in such a way that each $\Phi(V_{i_j})$ is contained in $\Phi(V)$ for a vertex $U\in V(\mathcal G_{M_1})$ for $j=1,\ldots,\ell$.
We now contract all vertices $\{V_{i_1},\ldots,V_{i_\ell}\}$ to obtain $U$.
With a similar method,  we are able to contract the finer partition corresponding~$\Phi(V(\mathcal G_{M_1}))$ to get the partition corresponding~$\Phi(V(\mathcal G_{M_2}))$ and again contract to get the partition corresponding~$\Phi(V(\mathcal G_{M_3}))$ or independently we can contract the partition corresponding~$\Phi(V(\mathcal G_{M_1}))$ to get~the partition corresponding~$\Phi(V(\mathcal G_{M_3}))$. 
This shows  that the above diagram is commutative. 
Now Lemma \ref{compact} completes the proof.
\end{proof}

\noindent Now we introduce the other family of inverse system. 
First, we define our auxiliary graphs.

\begin{deff}\label{contraction 1}
Let~$E\in \mathcal{P}(E(G))$ be a finite set. Then we remove~$E$ and we contract all vertices and edges of each component to a vertex.\footnote{So all ends which live on that component are corresponded to a vertex.} 
Now for every edge in~$E$, we join the corresponding vertices in the new graph.  We denote this new finite graph by~$\mathcal{G}.E$.
\end{deff}

\noindent Now we are ready to topologize~$\Vert\mathcal{G}.E\Vert$ for the auxiliary graph~$\mathcal{G}.E$. 
Every edge of~$\mathcal{G}.E$ is endowed by the topology of the closed unit interval~$[0,1]$.
The topology on~$\Vert\mathcal{G}.E\Vert$ is the coherent topology with all edges which is exactly the same as one complex topology here. 
Now suppose that~$E'\subseteq E$ are two finite subsets of~$E(G)$.\\
\noindent The definition of~$\mathcal{G}.E$ leads to a map~$f_{EE'}\colon\Vert\mathcal{G}.E\Vert\to \Vert\mathcal{G}.E'\Vert$. Every vertex of~$\mathcal{G}.E$ corresponds a component
of~$G\setminus E$. Thus this component is contained a component of~$G\setminus E'$ and so it defines~$f_{EE'}$ on vertices of the graph~$\mathcal{G}.E$. 
So we only need to define~$f_{EE'}$ on~$E\setminus E'$. 
Each of  these edges has to be a component of~$G\setminus E'$ and so~$f_{EE'}$ carries this edge to the corresponding vertex of its component.\\ 

\noindent Note that each~$\Vert\mathcal{G}.E\Vert$ is a compact Hausdorff space and it is not hard to see that~$f_{E'E}$ is continuous. An analogous argument of Theorem \ref{compatible} yields the following theorem.

\begin{thm}\label{lim G.m}
The system~$(\Vert\mathcal{G}.{E}\Vert,f_{EE'},\Gamma )$ is an inverse system, where~$\Gamma$ is the set of all finite sets of edges and~$E,E'\in \Gamma$ and moreover the space~$\underleftarrow{\lim}\,\Vert\mathcal{G}.E\Vert$ is compact. 
\end{thm}
 

\section{Reconstruction of Topologies}
\noindent In this section, we study connections between the following topological spaces: ($\Vert{G}\Vert\cup\Omega'(G)$, {\sc{Etop}}),   ($\Vert{G}\Vert\cup\Omega'(G),$ {\sc{FCtop}}), ($\widetilde{G}$, {\sc{IFCtop}}), ($\widetilde{G}$, {\sc{Itop}}),~$\underleftarrow{\lim}\,\Vert\mathcal{G}.E\Vert$ and~$\underleftarrow{\lim}\,\Vert\mathcal{G}_M\Vert$. 
In particular we show that the four last topological spaces are homeomorphic.

\noindent We start with the following theorem.
\begin{thm}\label{FCTop=ETop}
Let~$G$ be an arbitrary graph. Then topologies {\sc{FCtop}} and {\sc{Etop}} coincide.
\end{thm}
\begin{proof}
In order to show that they coincide, we have to prove that every basic closed set in {\sc{Etop}} is closed in {\sc{FCtop}} and vice versa.
First let~${\widetilde{G}\setminus \mathcal{O}_X(\omega)}$ be a basic closed set in {\sc{Etop}}, where~$\mathcal{O}_X(\omega)=F\cup\{\omega\in\Omega'(G)\mid \omega\text{ lives on }C\}\cup F'$ and~$X=\{x_1,\ldots,x_n\}$ are inner points of~$\{e_1,\ldots,e_n\}$ of edges such that~$x_i\in \mathring{e}_i$ for~$i=1,\ldots,n$ and $F$ is a component of $\widetilde{G}\setminus\{e_1,\ldots,e_n\}$ and $F'$ is a finite set of partial edges as we defined in Section 2.2.  
So we can suppose that~$\widetilde{G}\setminus \mathcal{O}_X(\omega)=F_1\cup\cdots\cup F_t$, where each~$F_i$ is the topological components of~$\widetilde{G}\setminus X$ except~$\mathcal{O}_X(\omega)$ and the corresponding inner points of~$X$. 
Without loss of generality, we can assume~$\{x_1,\ldots,x_t\}$ are inner points 
separating~$F_i$'s from~$\mathcal{O}_X(\omega)$. 
Now consider the edges containing  inner points 
$\{x_1,\ldots,x_r\}$, say~$C=\{e_1,\ldots,e_r\}$. 
Hence~$C$ forms a finite cut 
and each~$F_i$ is a topological component, for~$i=1,\ldots,t$. Thus every~$F_i$ is 
a closed set in~{\sc{FCtop}}, as desired.\\
Now suppose that~$C=(A,B)$ is an arbitrary finite cut and~$F_1,\ldots,F_s$ are components after removing~$\mathring{C}$. 
For a given~${\omega\in\Omega'(G)}$, let~$F_1$ be the component which contains a tail of a ray of~$\omega$ and let~$\mathcal{C}_C(\omega)$ be a basic closed set around~$\omega$. 
Assume that~$X=\{x_1,\ldots,x_s\}$ are arbitrary inner points of edges of~$C$ and let~$\mathcal{O}_X(\omega)$ be the corresponding basic open set containing~$F_1$. 
On the other hand, the union of all the other components of~$G\setminus X$ except~$\mathcal{O}_X(\omega)$ is an open set in {\sc{Etop}}. Thus the union of~$\left(G\setminus \mathcal{O}_X(\omega)\right)$ with~$\mathring{C}$ is an open set in {\sc{Etop}}. 
Therefore~$\mathcal{C}_C(\omega)$ is closed in {\sc{Etop}}, as it is the  complement of~$\left(G\setminus \mathcal{O}_X(\omega)\right)\cup \mathring{C}$.
\end{proof}

\noindent Now we have the following lemma.

\begin{lem} {\rm{\cite[Satz 2.1]{Schulz}}}
Let~$G$ be an infinite graph. Then~($\Vert{G}\Vert\cup\Omega'(G)$, {\sc{Itop}}) is a compact Hausdorff space.
\end{lem}

\noindent In Theorem \ref{fctopinverselimit}, we proved that~${\phi_M\colon (\Vert{G}\Vert\cup\Omega'(G),{\textit{\sc{FCtop}}})\to(\Vert\mathcal{G}_M\Vert, \tau_M)}$ is continuous. 
And since ${(\widetilde{G},{\textit{\sc{IFCtop}}})}$ is a quotient space of $(\Vert{G}\Vert\cup\Omega'(G),\textit{\sc{FCtop}})$, there is a continuous map from ${(\widetilde{G},{\textit{\sc{IFCtop}}})}$ to $(\Vert\mathcal{G}_M\Vert, \tau_M)$.
For the sake of simplicity, we call this map also $\phi_M$.

\begin{thm}\label{IFCTop=Itop}
Let~$G$  be an arbitrary infinite graph. Then the following topological spaces are homeomorphic.
\begin{enumerate}[\rm (i)]
\item~$(\widetilde{G},{\textit{\sc{Itop}}})$
\item~$(\widetilde{G},{\textit{\sc{IFCtop}}})$
\item~$\underleftarrow{\lim}\,\Vert\mathcal{G}_{M}\Vert$
\item~$\underleftarrow{\lim}\,\Vert\mathcal{G}.{E}\Vert$
\end{enumerate}
\end{thm}
\begin{proof}
 We prove it according to the following diagram:
\begin{center}
${\rm (i)}\Longleftrightarrow {\rm (ii)}\Longleftrightarrow {\rm (iii)}\Longleftrightarrow {\rm (iv)}~$
\end{center}
We show that topologies {\sc{Itop}} and~{\sc{IFCtop}} are equivalent on~$\widetilde{G}$.
We consider quotient maps~$\pi_1\colon(\Vert{G}\Vert\cup\Omega'(G),{\textit{\sc{Etop}}})\to (\widetilde{G},{\textit{\sc{Itop}}})$  and~$\pi_2\colon(\Vert{G}\Vert\cup\Omega'(G),{\textit{\sc{FCtop}}})\to (\widetilde{G},{\textit{\sc{IFCtop}}})$ and so Theorem \ref{FCTop=ETop} completes this part.\\
\noindent Now we show that~${\rm (ii)}\Longleftrightarrow {\rm (iii)}$. 
We use the universal property of the inverse limit of topological spaces. 
Let~$M_1$ and~$M_2$ be two finite subsets of~$\mathcal{B}_{\rm{fin}}(G)$ such that~$M_2\subseteq M_1$. 
Note that it follows from Lemma \ref{contraction map} and Theorem \ref{fctopinverselimit} that we have the continuous surjective maps~$\phi_{M_1}$ and~$\phi_{M_2}$  from~$\widetilde{G}$ to~$\mathcal{G}_{M_1}$ and~$\mathcal{G}_{M_2}$ respectively, as in the following commutative digram.
\begin{center}
\begin{tikzpicture}
\matrix (m) [matrix of math nodes,row sep=3em,column sep=2em,minimum width=2em] {
& \widetilde{G} & \\
&\underleftarrow{\lim}\,\Vert\mathcal{G}_{M}\Vert&\\
\Vert\mathcal{G}_{M_1}\Vert&  & ,\Vert\mathcal{G}_{M_2}\Vert \\};
\path[-stealth]
(m-1-2) edge node [left] {$h$} (m-2-2)
(m-2-2) edge (m-3-1)
(m-2-2) edge (m-3-3)
(m-1-2) edge [bend right=330]node [right] {$\phi_{M_2}$} (m-3-3)
(m-1-2) edge [bend left=330] node [left] {$\phi_{M_1}$} (m-3-1)
(m-3-1) edge node [above] {$\psi_{M_1M_2}$} (m-3-3);
\end{tikzpicture}
\end{center}
\begin{center}
Figure 4.1
\end{center}

\noindent By the universal property of the inverse limit, there is a  unique continuous map~${h\colon \widetilde{G}\to \underleftarrow{\lim}\,\Vert\mathcal{G}_{M}}\Vert$ and moreover since each~$\phi_M$ is surjective  for finite subset~$M\subseteq \mathcal{B}_{\rm{fin}}(G)$, it follows from~\cite[Corollary 1.1.6]{profinite group} that the map~$h$ is surjective. 
Next we show that~$h$ is injective as well. 
Assume that there are two distinct points~$x$ and~$y$ belonging to~$\widetilde{G}$ such that~$h([x])=h([y])$. 
There is a finite subset~$M\subseteq \mathcal{B}_{\rm{fin}}(G)$ such that the images~$x$ and~$y$ in~$\mathcal {G}_M$ are distinct and we deduce that the images~$[x]$ and~$[y]$ by~$h$ are different, as the above diagram is commutative. 
Note that Theorem \ref{fctopcompact} implies that~$\widetilde{G}$ with~{\sc{FCtop}} is a compact space and any continuous bijection from a compact space to a Hausdorff space is a homeomorphism map. 
So~$\widetilde{G}$ is homeomorphic to~$\underleftarrow{\lim}\,\Vert\mathcal{G}_{M}\Vert$, as desired.   

\noindent In order to show~${\rm (iii)}\Longleftrightarrow {\rm (iv)}$, we find a continuous bijection map between spaces~$\underleftarrow{\lim}\,\Vert\mathcal{G}_{M}\Vert$ and~$\underleftarrow{\lim}\,\Vert\mathcal{G}.{M}\Vert$. 
Since this map is a continuous map from a compact space to a Hausdorff space, this would complete the prove. 
Suppose that~$\Sigma$ and~$\Gamma$ are sets of finite subsets of~$E(G)$ and~$\mathcal{B}_{\rm{fin}}(G)$, respectively. Let~$E\in \Sigma$. 
Then assume that~$M$ is the set~$\{C\in\mathcal{B}_{\rm{fin}}(G)\mid C\subseteq E\}$. Let~$V$ be a vertex of~$\mathcal{G}_M$. 
Then every vertex of~$G$ in~$\Phi(V)$ belongs to a component of~$G\setminus E$.  
So we have a map~$g\colon \mathcal{G}_M\to\mathcal{G}.E$ where~$g$ carries each vertex~$V$ to the contraction of the corresponding component and~$g$ contracts every edge of~$\mathcal{G}_M$ out side of~$M$ to the vertex of~$\mathcal G.M$ corresponding with a suitable component of~$G\setminus E$. 
It is not hard to see that the extension of~$g\colon \Vert\mathcal{G}_M\Vert\to\Vert\mathcal{G}.E\Vert$ is a continuous map. 
Similarly, we define~$\Vert\mathcal{G}_{\widetilde{M}}\Vert$ and~$\tilde{g}$ for~$\widetilde{E}\subseteq E$. 
Thus we have the following commutative diagram:\\
\begin{center}
\begin{tikzpicture}
\matrix (m) [matrix of math nodes,row sep=3em,column sep=4em,minimum width=2em] {
\Vert\mathcal{G}_M\Vert & \Vert\mathcal{G}.E\Vert \\
\Vert\mathcal{G}_{\widetilde{M}}\Vert & \Vert\mathcal{G}.\widetilde{E}\Vert \\};
\path[-stealth]
(m-1-1) edge node [left] {$\psi_{M\widetilde{M}}$} (m-2-1)
edge  node [below] {$g$} (m-1-2)
(m-2-1.east|-m-2-2) edge node [below] {$\widetilde{g}$} node [above] {} (m-2-2)
(m-1-2) edge node [right] {$f_{E\widetilde{E}}$} (m-2-2);
\end{tikzpicture}
\end{center}
\begin{center}
Figure 4.2
\end{center}
\noindent Each~$g\colon \Vert\mathcal{G}_M\Vert\to\Vert\mathcal{G}.E\Vert$ induces a compatible continuous surjective map from $\underleftarrow{\lim}\,\Vert\mathcal{G}_{M}\Vert$ to~$\Vert\mathcal{G}.E\Vert$. 
It follows from \cite[Corollary 1.1.6]{profinite group} that the corresponding map which induces the mapping~$\theta\colon \underleftarrow{\lim}\,\Vert\mathcal{G}_{M}\Vert\to \underleftarrow{\lim}\,\Vert\mathcal{G}.{M}\Vert$ is a surjective continuous map. 
Now we show that~$\theta$ is  injective. 
Assume that~${x=(x_\alpha)_{\alpha\in I}}$ and~${y=(y_\alpha)_{\alpha\in I}}$ are distinct in~$\underleftarrow{\lim}\,\Vert\mathcal{G}_{M}\Vert$ such that~${\theta((x_\alpha)_{\alpha\in I})=\theta((y_\alpha)_{\alpha\in I})}$.  
Since~${x}$ and~$y$ are different, there is~$\alpha_0\in I$ such that~$x_{\alpha_0}$ and~$y_{\alpha_0}$ are different points in~$\Vert\mathcal{G}_{M_{\alpha_0}}\Vert$. 
Then there is a finite cut~$C$ in~$M_{\alpha_0}$ which separates~$x$ and~$y$. 
Thus the images~$x$ and~$y$ in~$\Vert\mathcal{G}.C\Vert$ are different and so~$\theta(x)\neq\theta(y)$. 
The other cases are similar. 
Thus we found a continuous bijection between~$\underleftarrow{\lim}\,\Vert\mathcal{G}_{M}\Vert$ and~$\underleftarrow{\lim}\,\Vert\mathcal{G}.{M}\Vert$.~\end{proof}

\section{Topological Spanning Trees in \sc{Itop}}
The aim of this section is to show how the auxiliaries graphs defined in the third section can be utilized to investigate topological spanning trees in~$(\widetilde{G},{\textit{\sc{Itop}}})$.
We first review some notations and definitions regarding topological spanning trees in~$(\widetilde{G},{\textit{\sc{Itop}}})$.

\noindent An \textit{arc} and a \textit{circle} in the space~$(\widetilde{G},{\textit{\sc{Itop}}})$ is a subspace homeomorphic to the closed interval~$[0,1]$ and the unit circle~$S^1$, respectively. 
A subspace~$H$ of~$\widetilde{G}$ is said a \textit{standard subspace} if it is the closure of some subgraph of~$G$.
\begin{deff}
A \textit{topological spanning tree} of~$\widetilde{G}$ is an arc-connected standard subspace~$T$ of~$\widetilde{G}$ that contains every vertex  of~$\widetilde{G}$ but contains no circle.
\end{deff}
\noindent We note that {\sc{Itop}} is obtained by taking quotient of {\sc Etop} and so {\sc{Itop}} is compact. since a topological spanning tree contains the class of every vertex of~$\widetilde{G}$, it should have every end as well.\\

\noindent We now need another terminology. 
A \textit{continuum} is a compact connected Hausdorff space.
\begin{lem}\label{arc=connected}{\rm\cite[Problem 6.3.11]{engelking}}
Every locally connected  metric continuum is arc-connected. 
\end{lem}
\noindent Now, suppose that $G$ is a countable graph and $H$ is a connected standard subspace of $(\widetilde{G},{\textit{\sc{Itop}}})$.
It follows from Theorem \ref{fctopcompact} that $H$ is compact.
Then by Theorem \ref{metrizable}, $H$ is metrizable. 
Thus $H$ is connected metric continuum and Lemma \ref{arc=connected} implies that $H$ is arc-connected.
If we summarize the above discussion, then we have the following theorem. 

\begin{thm}\label{arc=connected2}
If $G$ is countable, then every connected standard subspace of $(\widetilde{G},{\textit{\sc{Itop}}})$ is arc-connected.
\end{thm}
\noindent The following well-known lemma is important.
It can be found in \cite{diestel} with a different formulation.
\begin{lem}\label{arc-connected}
A standard subspace~$H$ of~$\widetilde{G}$ is arc-connected if and only if~$H$ contains an edge from every finite cut of~$G$ of which it meets both sides.
\end{lem}
\begin{proof}
First let~$H$ be arc-connected. 
Then assume to the contrary that~$G$ has a finite cut~$C=(A,B)$ which both~$A$ and~$B$ meet~$H$ such that~$H$ has no edge in~$C$.
Thus one can see that~$H\subseteq \widetilde{G}\setminus \mathring C$.
On the other hand, we know that~$\widetilde{G}\setminus \mathring C$ is equal to~$\overline{G[A]}\cup\overline{G[B]}$.
Now we claim that this union is disjoint.
Otherwise there is an element~$x\in \overline{G[A]}\cap\overline{G[B]}$.
Therefore the class of $x$ is the class of an end $\omega$ in $\widetilde{G}$, as $G[A]\cap G[B]=\emptyset$.
Pick up an arbitrary inner point from each edge in the finite cut~$C$.
Let~$\mathcal{O}(\omega)$ be a basic open neighbourhood of~$\omega$ in \textit{\sc{Etop}} containing~$G[B]$ with respect to these inner points.
Hence~$\mathcal{O}(\omega)$ has no intersection with~$G[A]$ and so the claim is proved.
Since~$\widetilde{G}$ is a disjoint union of closed sets,~$\widetilde{G}$ is not connected and so is not arc-connected and it yields a contradiction.\\
Conversely,  suppose that~$H=\overline{ (X,D)}$, where $X\subseteq V(G)$ and $D\subseteq E(G)$.
Assume to the contrary that $H$ is not arc-connected and equivalently by Theorem \ref{arc=connected2}, we can assume that  $H$ is not connected.
Let $H$ be the disjoint union of open sets $O_1$ and $O_2$ and set $X_i:=O_i\cap X$.
Let $C_1$ be a component of $X_1$ and let $P$ be a maximal edge-disjoint $C_1$-$X_2$ paths.
If there is a component of $X_2$ such that there are only finitely paths of $P$ between this component and $C_1$, then we have a finite cut between this component and $C_1$.
By the assumption, $H$ has to meet this finite cut and we get a contradiction with $H=O_1\cup O_2$.
Otherwise there are infinitely many paths between $C_1$ and each component of $X_2$.
Choose from each path a vertex.
So we have infinity many vertices.
It follows from Lemma \ref{star-comb} that $C_1$ contains either an end $\omega$ or a vertex $v$ with an infinite degree.
If $C_1$ has a vertex $v$ of the infinite degree, then we are not able to separate $v$ from each vertex of any component of $X_2$ and a contradiction is obtained.
So we can assume that there is an end which lives in $C_1$.
With a similar argument, we can show that any component of $X_2$ has an end.
Therefore there is an end belonging to $O_1$ and $O_2$ and  it yields a contradiction.
\end{proof}

\noindent A strategy for finding a topological spanning tree in~$\widetilde{G}$ is investigating spanning trees in each~$\mathcal{G}_M$ for every finite set~$M$ of ~$\mathcal{B}_{\rm{fin}}(G)$ and extending this spanning tree to a topological spanning tree in~$\widetilde{G}$. 
\begin{thm}
Let~$G$ be a countable graph. Then~$(\widetilde{G},{\textit{\sc{Itop}}})$ contains a topological spanning tree.
\end{thm}
\begin{proof}
We are going to construct trees in our inverse system inductively and we show that the limit of these trees is our required topological spanning tree.
Let~$M$ be a finite subset of ~$\mathcal{B}_{\rm{fin}}(G)$ and~$C\notin M$ be a finite cut.
We set~$M'=M\cup \{C\}$. 
Then we show that there exists a spanning tree~$T_M$ of~$\mathcal{G}_M$  such  that~$E(T_{M'})\cap E(\mathcal{G}_{M}) = E(T_M)$.
Suppose that~$V_1,\ldots,V_t$ are vertices of~$\mathcal{G}_M$.
Thus the set of~$\{\Phi(V_i)\mid i=1,\ldots,t\}$ is a partition of the vertex set of~$G$ and so we have~$\bigcup_{i=1}^t\Phi(V_i)=V(G)$.
Adding the cut $C$ refines the partition $\{\Phi(V_1),\ldots,\Phi(V_t)\}$.
We notice that $E(\mathcal{G}_{M})=E(\mathcal{G}_{M'})$.
Thus we are able to find the edges of $T_M$ in $\mathcal{G}_{M'}$.
We now add some edges of $\mathcal{G}_{M'}$ to $T_M$ to assure that we have a tree.
Let us denote the new tree with $T_M'$. 
We set $N_j=\{C_i\mid i\leq j\}$
Define$$T:=\overline{\bigcup_{i\in\mathbb N} E(T_{N_i})}.$$
We claim that~$T$ is a topological spanning tree of~$\widetilde{G}$.
In order to show that~$T$ is arc-connected, we invoke Lemma \ref{arc-connected}.
We have to show that every finite cut of~$G$ contains an edge from~$T$. 
By definition of the graphs~$\mathcal{G}_M$ and~$T_M$, we picked an edge up from each finite cut of~$M$ 
Next we show that~$T$ contains no circle.
Assume to contrary that~$T$ contains a circle~$C$.
Let~$u$,~$v$  be two vertices  of~$C$.
Then there exists a finite cut~$F$ separating~$u$ and~$v$.
So we can choose~$M$ large enough that~$M$ contains~$F$. 
Suppose that~$H$ is a fundamental cut with respect to $T_M$ of~$\mathcal G_M$ separating~$u$ from~$v$.
It is important to notice that~$H$ gives us a finite cut of~$G$.
Since~$C$ meets~$H$, it follows from Lemma \ref{arc-connected} that~$C$ should contain an edge from~$H$.
Let~$e\in C\cap H$.
Then since~$C\setminus e$ is still arc-connected, Lemma \ref{arc-connected} implies that~$C\setminus e$ meets~$H$.
Thus we can conclude that~$T$ has two edges in the finite cut~$H$.
Therefore we have a contradiction, as by definition one can see that~$E(T_M)=E(T)\cap E(\mathcal{G}_M)$ and we picked only one edge up from each finite cut of~$\mathcal{G}_M$.  
\end{proof}
\noindent We finish our paper with the following finial remark.
\begin{remark}
We have defined {\sc{Itop}} as the quotient topology of {\sc{Etop}} and in the above we constructed a topological spanning in $\widetilde{G}$ as a limit of spanning trees.
In our proof, we benefit so much from the properties of {\sc{Etop}} and we cannot replace it with the others topologies.
For instance, if we apply the quotient topology on {\sc{Top}}, not necessarily there is a topological spanning tree on $\widetilde{G}$.
Diestel and K$\ddot{u}$hn have discovered a counterexample that shows that the quotient topology of {\sc{Top}} does not contain any topological spanning tree, see~{\rm\cite[Corollary 3.5]{diestelkuhn2}}.
\end{remark}

\vspace{.05cm}
\noindent{\bf Acknowledgements.}  The author is so grateful to Nathan Bowler, Pascal Gollin, Matthias Hamann and Tim R\"{u}hmann for their comments and the invaluable discussions during the preparation of this paper.
{}  
  
   \end{document}